\documentclass[12pt]{amsart}
\usepackage{amsmath,amsthm,amsfonts,amssymb,latexsym,enumerate,url,hyperref} 

\begin{document}

\theoremstyle{plain}

\newtheorem{thm}{Theorem}[section]
\newtheorem{lem}[thm]{Lemma}
\newtheorem{conj}[thm]{Conjecture}
\newtheorem{pro}[thm]{Proposition}
\newtheorem{cor}[thm]{Corollary}
\newtheorem{que}[thm]{Question}
\newtheorem{rem}[thm]{Remark}
\newtheorem{defi}[thm]{Definition}
\newtheorem{hyp}[thm]{Hypothesis}

\newtheorem*{thmA}{THEOREM A}
\newtheorem*{corB}{COROLLARY B}

\newtheorem*{thmC}{THEOREM C}
\newtheorem*{conjA}{CONJECTURE A}
\newtheorem*{conjB}{CONJECTURE B}
\newtheorem*{conjC}{CONJECTURE C}

\newtheorem*{thmAcl}{Main Theorem$^{*}$}
\newtheorem*{thmBcl}{Theorem B$^{*}$}

\numberwithin{equation}{section}

\newcommand{\Maxn}{\operatorname{Max_{\textbf{N}}}}
\newcommand{\Syl}{\operatorname{Syl}}
\newcommand{\dl}{\operatorname{\mathfrak{d}}}
\newcommand{\Con}{\operatorname{Con}}
\newcommand{\cl}{\operatorname{cl}}
\newcommand{\Stab}{\operatorname{Stab}}
\newcommand{\Aut}{\operatorname{Aut}}
\newcommand{\Ker}{\operatorname{Ker}}
\newcommand{\IBr}{\operatorname{IBr}}
\newcommand{\Irr}{\operatorname{Irr}}
\newcommand{\SL}{\operatorname{SL}}
\newcommand{\FF}{\mathbb{F}}
\newcommand{\NN}{\mathbb{N}}
\newcommand{\N}{\mathbf{N}}
\newcommand{\C}{\mathbf{C}}
\newcommand{\OO}{\mathbf{O}}
\newcommand{\F}{\mathbf{F}}
\newcommand{\aplii}[4]{\begin{array}{ccc}
	#1 &\longrightarrow & #2\\
	#3 &\longmapsto & #4
	\end{array}}

\renewcommand{\labelenumi}{\upshape (\roman{enumi})}

\newcommand{\GL}{\operatorname{GL}}
\newcommand{\Sp}{\operatorname{Sp}}
\newcommand{\PGL}{\operatorname{PGL}}
\newcommand{\PSL}{\operatorname{PSL}}
\newcommand{\SU}{\operatorname{SU}}
\newcommand{\PSU}{\operatorname{PSU}}
\newcommand{\PSp}{\operatorname{PSp}}

\providecommand{\V}{\mathrm{V}}
\providecommand{\E}{\mathrm{E}}
\providecommand{\ir}{\mathrm{Irr_{rv}}}
\providecommand{\Irrr}{\mathrm{Irr_{rv}}}
\providecommand{\re}{\mathrm{Re}}

\def\irrp#1{{\rm Irr}_{p'}(#1)}
\def\ibrp#1{{\rm IBr}_{p'}(#1)}
\def\lin#1{{\rm Lin}(#1)}

\def\Z{{\mathbb Z}}
\def\C{{\mathbb C}}
\def\Q{{\mathbb Q}}
\def\irr#1{{\rm Irr}(#1)}
\def\ibr#1{{\rm IBr}(#1)}
\def\irra#1{{\rm Irr}_{\rm A}(#1)}
\def\ibra#1{{\rm IBr}_{\rm A}(#1)}
\def \c#1{{\cal #1}}
\def\cent#1#2{{\bf C}_{#1}(#2)}
\def\syl#1#2{{\rm Syl}_#1(#2)}
\def\nor{\trianglelefteq\,}
\def\oh#1#2{{\bf O}_{#1}(#2)}
\def\Oh#1#2{{\bf O}^{#1}(#2)}
\def\zent#1{{\bf Z}(#1)}
\def\det#1{{\rm det}(#1)}
\def\ker#1{{\rm ker}(#1)}
\def\norm#1#2{{\bf N}_{#1}(#2)}
\def\alt#1{{\rm Alt}(#1)}
\def\iitem#1{\goodbreak\par\noindent{\bf #1}}
   \def \mod#1{\, {\rm mod} \, #1 \, }
\def\sbs{\subseteq}

\def\gc{{\bf GC}}
\def\ngc{{non-{\bf GC}}}
\def\ngcs{{non-{\bf GC}$^*$}}
\newcommand{\notd}{{\!\not{|}}}
\newcommand{\Out}{{\mathrm {Out}}}
\newcommand{\Mult}{{\mathrm {Mult}}}
\newcommand{\Inn}{{\mathrm {Inn}}}
\newcommand{\IBR}{{\mathrm {IBr}}}
\newcommand{\IBRL}{{\mathrm {IBr}}_{\ell}}
\newcommand{\IBRP}{{\mathrm {IBr}}_{p}}
\newcommand{\ord}{{\mathrm {ord}}}
\def\id{\mathop{\mathrm{ id}}\nolimits}
\renewcommand{\Im}{{\mathrm {Im}}}
\newcommand{\Ind}{{\mathrm {Ind}}}
\newcommand{\diag}{{\mathrm {diag}}}
\newcommand{\soc}{{\mathrm {soc}}}
\newcommand{\End}{{\mathrm {End}}}
\newcommand{\sol}{{\mathrm {sol}}}
\newcommand{\Hom}{{\mathrm {Hom}}}
\newcommand{\Mor}{{\mathrm {Mor}}}
\newcommand{\St}{{\sf {St}}}
\def\rank{\mathop{\mathrm{ rank}}\nolimits}
\newcommand{\Tr}{{\mathrm {Tr}}}
\newcommand{\tr}{{\mathrm {tr}}}
\newcommand{\Gal}{{\it Gal}}
\newcommand{\Spec}{{\mathrm {Spec}}}
\newcommand{\ad}{{\mathrm {ad}}}
\newcommand{\Sym}{{\mathrm {Sym}}}
\newcommand{\Char}{{\mathrm {char}}}
\newcommand{\pr}{{\mathrm {pr}}}
\newcommand{\rad}{{\mathrm {rad}}}
\newcommand{\abel}{{\mathrm {abel}}}
\newcommand{\codim}{{\mathrm {codim}}}
\newcommand{\ind}{{\mathrm {ind}}}
\newcommand{\Res}{{\mathrm {Res}}}
\newcommand{\Ann}{{\mathrm {Ann}}}
\newcommand{\Ext}{{\mathrm {Ext}}}
\newcommand{\Alt}{{\mathrm {Alt}}}
\newcommand{\AAA}{{\sf A}}
\newcommand{\SSS}{{\sf S}}
\newcommand{\CC}{{\mathbb C}}
\newcommand{\CB}{{\mathbf C}}
\newcommand{\RR}{{\mathbb R}}
\newcommand{\QQ}{{\mathbb Q}}
\newcommand{\ZZ}{{\mathbb Z}}
\newcommand{\KK}{{\mathbb K}}
\newcommand{\NB}{{\mathbf N}}
\newcommand{\ZB}{{\mathbf Z}}
\newcommand{\OB}{{\mathbf O}}
\newcommand{\EE}{{\mathbb E}}
\newcommand{\PP}{{\mathbb P}}
\newcommand{\GC}{{\mathcal G}}
\newcommand{\HC}{{\mathcal H}}
\newcommand{\AC}{{\mathcal A}}
\newcommand{\BC}{{\mathcal B}}
\newcommand{\GA}{{\mathfrak G}}
\newcommand{\SC}{{\mathcal S}}
\newcommand{\TC}{{\mathcal T}}
\newcommand{\DC}{{\mathcal D}}
\newcommand{\LC}{{\mathcal L}}
\newcommand{\RC}{{\mathcal R}}
\newcommand{\CL}{{\mathcal C}}
\newcommand{\EC}{{\mathcal E}}
\newcommand{\GCD}{\GC^{*}}
\newcommand{\TCD}{\TC^{*}}
\newcommand{\FD}{F^{*}}
\newcommand{\GD}{G^{*}}
\newcommand{\HD}{H^{*}}
\newcommand{\hG}{\hat{G}}
\newcommand{\hP}{\hat{P}}
\newcommand{\hQ}{\hat{Q}}
\newcommand{\hR}{\hat{R}}
\newcommand{\GCF}{\GC^{F}}
\newcommand{\TCF}{\TC^{F}}
\newcommand{\PCF}{\PC^{F}}
\newcommand{\GCDF}{(\GC^{*})^{F^{*}}}
\newcommand{\RGTT}{R^{\GC}_{\TC}(\theta)}
\newcommand{\RGTA}{R^{\GC}_{\TC}(1)}
\newcommand{\Om}{\Omega}
\newcommand{\eps}{\epsilon}
\newcommand{\varep}{\varepsilon}
\newcommand{\al}{\alpha}
\newcommand{\chis}{\chi_{s}}
\newcommand{\sigmad}{\sigma^{*}}
\newcommand{\PA}{\boldsymbol{\alpha}}
\newcommand{\gam}{\gamma}
\newcommand{\lam}{\lambda}
\newcommand{\la}{\langle}
\newcommand{\ra}{\rangle}
\newcommand{\hs}{\hat{s}}
\newcommand{\htt}{\hat{t}}
\newcommand{\sgn}{\mathsf{sgn}}
\newcommand{\SR}{^*R}
\newcommand{\tn}{\hspace{0.5mm}^{t}\hspace*{-0.2mm}}
\newcommand{\ta}{\hspace{0.5mm}^{2}\hspace*{-0.2mm}}
\newcommand{\tb}{\hspace{0.5mm}^{3}\hspace*{-0.2mm}}
\def\skipa{\vspace{-1.5mm} & \vspace{-1.5mm} & \vspace{-1.5mm}\\}
\newcommand{\tw}[1]{{}^#1\!}
\renewcommand{\mod}{\bmod \,}

\newcommand{\Irre}[1]{{\rm Irr}(#1)}
\newcommand{\Irra}[2]{{\rm Irr}_{#1}(#2)}
\renewcommand{\N}[2]{{{\bf N}_{#1}(#2)}}
\def\st{\text{ }|\text{ }}
\newcommand{\inv}[1]{{#1}^{-1}}

\marginparsep-0.5cm

\renewcommand{\thefootnote}{\fnsymbol{footnote}}
\footnotesep6.5pt

\title{McKay Bijections and Decomposition Numbers}
\author{David Cabrera-Berenguer}
\address{Departament de Matem\`atiques, Universitat de Val\`encia, 46100 Burjassot,
Val\`encia, Spain}
\email{david.cabrera@uv.es}

\thanks{This research is supported by Grant PID2022-137612NB-I00
 funded by MCIN/AEI/ 10.13039/501100011033 and ERDF ``A way of making Europe”. The author would like to thank
Gabriel Navarro, J. M. Martínez and Noelia Rizo for useful conversations on the subject.
}

\keywords{}

\subjclass[2010]{Primary 20C20; Secondary 20C15}

\begin{abstract}
If $G$ is $p$-solvable, we prove that there exists a McKay bijection that respects 
the decomposition numbers $d_{\chi\varphi}$, whenever $\varphi$ is linear.
 \end{abstract}

\maketitle

\section{Introduction} 
The McKay conjecture (now a theorem, see \cite{CS24}) establishes that there exists a bijection
$$f: \irrp G \rightarrow \irrp{\norm GP} \, ,$$
where $G$ is a finite group, $p$ is a prime, $P\in \syl pG$, and $\irrp G$ is the set of the irreducible
complex characters of $G$ of degree not divisible by $p$. It has also been conjectured that $f$ can be chosen to commute with the action of ${\rm Aut}(G)_P \times {\mathcal G}_{\mathcal P}$,
where ${\rm Aut}(G)_P$ is the group of automorphisms of $G$ that fix $P$ and 
${\mathcal G}_{\mathcal P}$ is the group of Galois automorphisms of ${\rm Gal}(\Q_{|G|})$ that fix
any prime ideal $\mathcal P$ of the ring of algebraic integers in the cyclotomic field $\Q_{|G|}$
(see Conjectures 9.18 and 9.13 in \cite{N18} and \cite{N04}).
Furthermore, $f$ should satisfy that
$f(\chi)(1) \equiv \pm\chi(1) $ mod $p$ for $\chi \in \irrp G$ according to \cite{IN02}.

If $G$ is $p$-solvable, but not in general, $f$ can also be chosen satisfying that $f(\chi)(1)$ divides $\chi(1)$ (see \cite{R19} and \cite{T07}). 
Without $p$-solvability, E. Giannelli has proposed that $f$ can be chosen such that $f(\chi)(1) \le \chi(1)$ for $\chi \in \irrp G$ (see \cite{G25}).
\medskip

What about $p$-decomposition numbers? Recall that if $\chi \in \irr G$ and $\ibr G$ is a set of $p$-irreducible Brauer characters, we have that
$$\chi^0=\sum_{\varphi \in \ibr G} d_{\chi \varphi} \varphi\, ,$$
for unique non-negative numbers called the {\sl decomposition numbers}. Also, the {\sl projective indecomposable} character associated with $\varphi\in\ibr G$ is defined as
\[\Phi_{\varphi}=\sum_{\chi\in\irr G}d_{\chi\varphi}\chi.\] The problem of relating McKay bijections with decomposition numbers seems to be quite hard, if not impossible.

\medskip
The purpose of this note is to prove the following prediction of G. Navarro.

\begin{thmA}
Suppose that $G$ is a $p$-solvable finite group and let $P\in\syl p G$. Then there exists a
bijection $f: \irrp G \rightarrow \irrp{\norm GP}$ such that $$d_{\chi \varphi}=d_{f(\chi) \varphi_{\norm GP}}$$
for $\chi \in \irrp G$ and $\varphi \in \ibr G$ linear.
\end{thmA}

Notice that if $P=\norm GP$ and $\varphi$ is trivial, then Theorem A extends the $p$-solvable case of Theorem B
in \cite{NT}.
In general, we cannot expect the equality in Theorem A to hold outside $p$-solvable groups
(see for instance ${\sf A}_5$ with $p=2$),
 and even in groups with a $p$-complement 
(see for instance ${\rm PSL}_2(8)$ with $p=3$). In these two groups, we have that   $d_{\chi 1} \ge  d_{\chi^* 1}$ holds
for $\chi \in \irrp G$, but unfortunately this inequality does not hold in general as 
shown by $M_{24}$ for $p=2$, where the Sylow
2-subgroup is self-normalizing and yet
it has odd-degree characters $\chi$ with $d_{\chi1}=0$. Since in $p$-solvable groups there are also
bijections $^*: \ibrp G \rightarrow \ibrp{\norm GP}$, where $\ibrp G$ is the set of $p'$-degree $p$-Brauer characters of $G$ (see Theorem A of \cite{W90}), it is somewhat reasonable to ask if bijections can be chosen such that $d_{\chi \varphi}=
d_{\chi^* \varphi^*}$. The answer for this is negative, at least for non-linear Brauer characters
$\varphi$: for instance in the group ${\tt SmallGroup}(216,153)$,
for $p=3$.  At the time of this writing, however, we haven't found an example showing that 
there are not bijections satisfying $d_{\chi \varphi} \ge
d_{\chi^* \varphi^*}$  in $p$-solvable groups.

\medskip
Theorem A implies the following, which does not seem to have been noticed before.

  \medskip
  \begin{corB}
  Let $G$ be $p$-solvable and let $\chi \in \irrp G$. Then $d_{\chi 1}=1$ or $0$.
  The number of $\chi \in \irrp G$ with $d_{\chi 1}=1$ is the number of $\norm GP$-orbits on
  $P/P'$, where $P \in \syl p G$.
  \end{corB}
It is not difficult to prove that, as a consequence of the McKay conjecture, $|\irrp G|$ is the number of $\N G P$-orbits on $P/P'$ if and only if $\N G P=P$. Hence, Corollary B is consistent with Conjecture A of \cite{NT}.
\section{Proofs}
Our notation for characters follows \cite{Is} and \cite{N18}. Our notation for modular characters follows \cite{N98}. 
We begin by quoting the results that we shall need for the reader's convenience.

\begin{lem}\label{restBij}
Suppose that $G$ is a finite group, $N\unlhd G$, and $H\leq G$ is such that $G=NH$. Let $M=N\cap H$. Then the restriction map ${\rm Char}(G/N)\to{\rm Char}(H/M)$ is a bijection satisfying
\[[\alpha,\beta]=[\alpha_H,\beta_H]\]
for $\alpha,\beta\in{\rm Char}(G/N)$. Hence, the restriction defines a bijection $\Irr(G/N)\to\Irr(H/M)$.
\end{lem}
\begin{proof}
    See Theorem (1.18) of \cite{N18}.
\end{proof}
\begin{lem}[Gallagher correspondence]
    Suppose that $G$ is a finite group, $N\unlhd G$ and $\theta\in\irr N$ extends to some $\chi\in\irr G$. Then the map \[\aplii{\irr{G/N}}{\irr{G|\theta}}{\beta}{\beta\chi}\]
    is a bijection.
\end{lem}
\begin{proof}
    See Corollary (6.17) of \cite{Is}.
\end{proof}
\begin{lem}\label{extSylow}
    Let $N\unlhd G$ and let $\theta\in\Irr(N)$ be $G$-invariant. Then $\theta$ extends to $G$ if and only if $\theta$ extends to $P$ for every Sylow subgroup $P/N$ of $G/N$.
\end{lem}
\begin{proof}
    See Theorem (5.10) of \cite{N18}.
\end{proof}
\begin{lem}\label{PinvConst}
    Let $G$ be a finite group, let $P\in\syl p G$ and let $L\unlhd G$. If $\chi\in\irrp G$, then $\chi_L$ has a $P$-invariant constituent and any two of them are $\norm G P$-conjugate.
\end{lem}
\begin{proof}
    See Lemma (9.3) of \cite{N18}.
\end{proof} 
Before stating the next preliminary result, we fix some notation. For a fixed prime $p$, we denote by $G^0$ the set of elements of $G$ whose orders are not divisible by $p$. Also, if $\Psi$ is a (complex) class function of $G$, we denote its restriction to $G^0$ by $\Psi^0$.
\begin{lem}\label{extensionBrauer}
    Let $G$ be a finite group and let $N=G'${\rm \textbf{O}}$^{p'}(G)$ be the smallest normal subgroup of $G$ whose quotient is an abelian $p'$-group. Then the map 
    \[\aplii{\{\chi\in\irr G:N\subseteq\ker\chi\}}{\{\varphi\in\ibr G:\varphi(1)=1\}}{\chi}{\chi^0}\] 
    is a bijection.
\end{lem}
\begin{proof}
    See Problem (2.7) of \cite{N98}.
\end{proof}
\begin{lem}\label{weakGlaub}
    Let $A$ be a finite group acting coprimely on a finite group $G$. If $\alpha,\beta\in\irr G$ are linear, $A$-invariant and $\alpha_{\textbf{C}_G(A)}=\beta_{\textbf{C}_G(A)}$, then $\alpha=\beta$.
\end{lem}
\begin{proof}
    Working with $\alpha\inv{\beta}$ we may assume that $\alpha_{\textbf{C}_G(A)}=1$ and prove that $\alpha$ is trivial. As $\alpha$ is linear and $A$-invariant then $[G,A]\subseteq\ker\alpha$, and therefore $[G,A]\textbf{C}_G(A)\subseteq\ker\alpha$. By coprime action it holds that $G=[G,A]\textbf{C}_G(A)$, and the result follows.
\end{proof}
Recall that if $N\unlhd G$ and $\theta\in\irr N$, then $\irr{G|\theta}$ denotes the set of irreducible characters of $G$ which lie over $\theta$. Also, for a fixed prime $p$, $\irrp{G|\theta}$ denotes the subset of irreducible characters of $\irr{G|\theta}$ whose degree is not divisible by $p$.

We can now prove Theorem A.
\begin{thm}
Suppose that $G$ is a $p$-solvable finite group and let $P\in\syl p G$. Then there exists a
bijection $f: \irrp G \rightarrow \irrp{\norm GP}$ such that $$d_{\chi \varphi}=d_{f(\chi) \varphi_{\norm GP}}$$
for $\chi \in \irrp G$ and $\varphi \in \ibr G$ linear.
\end{thm}
\begin{proof}
 Let $H$ be a $p$-complement of $G$. If $\varphi\in\ibr G$ is linear, then by Problem 2.8 of \cite{N98},
 we have that $$(\varphi_H)^G=\Phi_{\varphi} \, ,$$
 where $\Phi_\varphi$ is the projective indecomposable character associated with $\varphi$.
 Therefore, if $\chi \in \irr G$, then we have that
 $$d_{\chi \varphi}=[\chi_H, \varphi_H]\, .$$
We proceed by induction over $|G|$. First we suppose that $N=\OO_p(G)>1$. Let $\Delta$ be a complete set of representatives of the $\N G P$-action on the linear characters of $N$ that extend to $P$, and let $\theta\in\Delta$.
    As $G=G_\theta H$, then $|G_\theta:G_\theta\cap H|$ is a $p$-power and hence $(G_\theta\cap H)N/N$ is a $p$-complement of $G_\theta/N$. Also, $\N{G_\theta\cap H}PN/N$ is a $p$-complement of $\N{G_\theta} P/N$. By induction, there exists a bijection $\tilde f_\theta:\Irra{p'}{G_\theta/N}\to\Irra{p'}{\N{G_\theta} P/N}$ satisfying 
    \[[\hat\beta_{(G_\theta\cap H)N/N},\hat\varphi_{(G_\theta\cap H)N/N}]=[\tilde f_\theta(\hat\beta)_{\N{G_\theta\cap H}PN/N},\hat\varphi_{\N{G_\theta\cap H}PN/N}]\]
    for every $\hat\beta\in\irrp{G_\theta/N}$ and every linear $\hat\varphi\in\ibr{G_\theta/N}$.
    For every $N\unlhd G$, the map
    \[^\sim:\{\chi\in{\rm Char}(G):N\subseteq\ker\chi\}\to{\rm Char}(G/N)\]
    defined via $\tilde\chi(Ng)=\chi(g)$ is a bijection preserving the inner product. We denote its inverse by $r$. Sometimes it is convenient to identify $\irr{G/N}$ as a subset of $\irr G$, as we do next.
    For each $\beta\in\irrp{G_\theta/N}\subseteq\irr{G_\theta}$ we define $\hat f_\theta(\beta)=r(\tilde f_\theta(\tilde\beta))$. 
    
    We prove that $\theta$ extends to $G_\theta$. By Lemma \eqref{extSylow} it suffices to check that $\theta$ extends to $Q$ for every Sylow subgroup $Q/N$ of $G_\theta/N$. Let $Q/N\in\syl q {G_\theta/N}$. If $q=p$, then $\theta$ extends to $Q$ by the definition of $\Delta$. If $q\not=p$, then $\theta(1)o(\theta)$ is coprime to $|Q:N|$, and therefore by Corollary 6.27 of \cite{Is}, it follows that $\theta$ extends to $Q$. Thus, we conclude that $\theta$ extends to $G_\theta$. 
    
    Let $\gamma$ be an extension of $\theta$ to $G_\theta$. As its $p'$-part satisfies $(\gamma_{p'})_N=1_N$, then we may assume that $o(\gamma)$ is a $p$-power.
    Consider \[h:\Irra{p'}{G_\theta/N}\to\Irra{p'}{G_\theta|\theta},\] \[j:\Irra{p'}{\N{G_\theta}P/N}\to\Irra{p'}{\N{G_\theta}P|\theta}\] to be the respective Gallagher correspondences associated to $\gamma$ and $\gamma_{\N{G_\theta}P}$, and define $g=\inv h\hat f_\theta j$. 
    Let \[a:\Irra{p'}{G_\theta|\theta}\to\Irra{p'}{G|\theta},\] \[b:\Irra{p'}{\N{G_\theta}P|\theta}\to\Irra{p'}{\N GP|\theta}\] be the respective Clifford correspondences and let $f_\theta=\inv agb$. We prove that $f_\theta$ satisfies the desired condition. Let $\chi\in\Irra{p'}{G|\theta}$ and let $\varphi\in\ibr G$ be linear. Let $\psi\in\Irra{p'}{G_\theta|\theta}$ be the Clifford correspondent of $\chi$ with respect to $\theta$. Since $G=G_\theta H$, by Mackey's theorem we may write $\chi_H=(\psi^{G_\theta H})_H=(\psi_{G_\theta\cap H})^H$. Also, since $\N G P=\N{G_\theta}{P}\N H P$, then $(g(\psi)_{\N{H\cap G_\theta}P})^{\N H P}=(g(\psi)^{\N G P})_{\N H P}$. Now, let $\beta\in\Irra{p'}{G_\theta/N}$ be such that $\beta\gamma=\psi$ and let $\tau\in\irr G$ extending $\varphi_H$ such that $N\subseteq\ker\tau$ (this extension exists by Lemma \eqref{extensionBrauer}, as $\OO ^{p'}(G)$ is generated by the $p$-elements of $G$). Thus, by Lemma \eqref{restBij} it follows
    
    \[\begin{aligned}
        {}[\chi_H,\varphi_H]&=[\psi_{G_\theta\cap H},\varphi_{G_\theta\cap H}]\\&=[\beta_{G_\theta\cap H}\gamma_{G_\theta\cap H},\varphi_{G_\theta\cap H}]\\&=[\beta_{G_\theta\cap H},\tau_{G_\theta\cap H}]\\
        &=[\beta_{(G_\theta\cap H)N},\tau_{(G_\theta\cap H)N}]\\
        &=[\widetilde{\beta_{(G_\theta\cap H)N}},\widetilde{\tau_{(G_\theta\cap H)N}}]\\
        &=[(\tilde\beta)_{(G_\theta\cap H)N/N},(\tilde\tau)_{(G_\theta\cap H)N/N}]\\
        &=[(\tilde\beta)_{(G_\theta\cap H)N/N},(((\tilde\tau)_{G_\theta/N})^0)_{(G_\theta\cap H)N/N}].\\
        \end{aligned}\]
        Now, since $((\tilde\tau)_{G_\theta/N})^0$ is a linear Brauer character of $G_\theta/N$ it follows that
        \[\begin{aligned}
        {}[\chi_H,\varphi_H]&=[\tilde f_\theta(\tilde\beta)_{\N{G_\theta\cap H}PN/N},(\tilde\tau)_{\N{G_\theta\cap H}PN/N}]\\
        &=[r(\tilde f_\theta(\tilde\beta))_{\N{G_\theta\cap H}PN},\tau_{\N{G_\theta\cap H}PN}]\\
        &=[\hat f_\theta(\beta)_{\N{G_\theta\cap H}P},\tau_{\N{G_\theta\cap H}P}]\\
        &=[\hat f_\theta(\beta)_{\N{G_\theta\cap H}P},\varphi_{\N{G_\theta\cap H}P}]\\
    &=[\hat f_\theta(\beta)_\N{G_\theta\cap H}{P}\gamma_\N{G_\theta\cap H}{P},\varphi_\N{G_\theta\cap H}{P}]\\&=[g(\psi)_\N{G_\theta\cap H}{P},\varphi_\N{G_\theta\cap H}{P}]
        \\&=[(g(\psi)_\N{G_\theta\cap H}{P})^{\N H P},\varphi_\N H P)]\\
        &=[(g(\psi)^{\N G P})_{\N H P},\varphi_{\N H P}]\\&=[f_\theta(\chi)_{\N H P}, \varphi_{\N H P}].
    \end{aligned}\]
    Thus, our bijection $f_\theta:\irrp{G|\theta}\to\irrp{\N GP|\theta}$ satisfies that $d_{\chi\varphi}=d_{f_\theta(\chi)\varphi_{\N G P}}$ for every $\chi\in\irrp G$ and each linear $\varphi\in\ibr G$, as desired.
    
    Now let $\theta_1,\ldots,\theta_n$ be a complete set of representatives of the orbits of $\N G P$ on the linear $P$-invariant characters of $N$. Then, by Lemma \eqref{PinvConst} we may write 
    $$\Irra{p'}{G}=\bigcup_{i=1}^n\Irra{p'}{G|\theta_i}$$ and $$\Irra{p'}{\N G P}=\bigcup_{i=1}^n\Irra{p'}{\N G P|\theta_i}$$ as disjoint unions. Let $\chi\in\Irra{p'}G$. Then there exists a unique $1\leq i\leq n$ such that $\chi\in\Irra{p'}{G|\theta_i}$. We prove that, in fact, $\theta=\theta_i\in\Delta$. Let $\psi\in\irrp{G_\theta|\theta}$ be the Clifford correspondent of $\chi$ over $\theta$. Then $\psi_P$ contains some linear irreducible constituent, which necessarily extends $\theta$.
    
    We define $f(\chi)=f_{\theta_i}(\chi)$. This is a bijection satisfying the condition.
    Hence we may assume that $N=1$, and therefore $K=\OO_{p'}(G)>1$.
    
Let $\Delta$ be the set of linear characters of $K$ which extend to $G$. Let $\theta\in\Delta$ and fix an extension $\gamma\in\irr G$ of $\theta$. Let \[h:\Irra{p'}{G/K}\to\Irra{p'}{G|\theta},\] \[j:\Irra{p'}{\N GP/\N K P}\to\Irra{p'}{\N G P|\theta_{\N K P}}\] be the respective Gallagher correspondences associated to $\gamma$ and $\gamma_{\N G P}$.

As $\N{G/K}{PK/K}=\N GPK/K$, then by induction there exists a bijection $\hat f:\Irra{p'}{G/K}\to\Irra{p'}{\textbf{N}_G(P}K/K)$  such that 
    \[[\beta_{H/K},\varphi_{H/K}]=[\hat f(\beta)_{\N HPK/K},\varphi_{\N HPK/K}]\]
    for every $\beta\in\Irra{p'}{G/K}$ and each linear $\varphi\in\ibr{G/K}$.
    With the usual identifications for $\Irre{G/K}$ and $\Irre{\N G PK/K}$ it follows that for every $\beta\in\Irra{p'}{G/K}$ we have
    \[[\beta_H,\varphi_H]=[\hat f(\beta)_{\N H PK},\varphi_{\N H PK}]=[\hat f(\beta)_\N H P,\varphi_\N H P],\]
    where the last equality holds by Lemma \eqref{restBij}.
    By the same lemma, the restriction $r:\irrp{\N GPK/K}\to\irrp{\N G P/\N K P}$ defines a bijection. We define $f_\theta=h^{-1}\hat frj$.

    By elementary group theory, $\N{KP}P=\textbf{C}_K(P)P$, and hence $\N K P=\textbf{C}_K(P)$. By Lemma \eqref{weakGlaub} the map $\Delta\to\irr{\N K P}$ defined via $\theta\mapsto \theta_{\N K P}$ is injective. Then, the unions
    \[\bigcup_{\theta\in\Delta}\irrp{G|\theta},\]
    \[\bigcup_{\theta\in\Delta}\irrp{\N G P|\theta_{\N K P}}\]
    are disjoint by Clifford's theorem. Therefore, by the $p$-solvable case of the McKay conjecture there exists $f:\irrp G\to\irrp{\N G P}$ such that for every $\theta\in\Delta$ the restriction of $f$ to $\irrp{G|\theta}$ is $f_\theta$. We prove that $f$ satisfies the desired condition.

    Let $\chi\in\irrp G$ and let $\varphi\in\ibr G$ be linear. By Lemma \eqref{extensionBrauer} we may take $\tau\in\irr G$ such that $\tau_H=\varphi_H$. Note that, in particular, $\varphi_K\in\Delta$. By the definition of $f$, $\chi$ lies over $\varphi_K$ if and only if $f(\chi)$ lies over $\varphi_{\N K P}$. Therefore, if $\chi$ does not lie over $\varphi_K$ then
    \[0=[\chi_H,\varphi_H]=[f(\chi)_{\N H P}, \varphi_{\N H P}].\]
    Then we may suppose that $\chi$ lies over $\varphi_K$. By the Gallagher correspondence there is $\alpha\in\irr {G/K}$ linear such that $\tau=\alpha\gamma$ for our prefixed extension $\gamma$ of $\varphi_K$. Also, there exists $\beta\in\irr{G/K}$ such that $\chi=\beta\gamma$. Since $\alpha^0\in\ibr{G/K}$ then
    \[\begin{aligned}
        {}[\chi_H,\varphi_H]&=[\beta_H\gamma_H,\tau_H]\\
        &=[\beta_H\gamma_H,\alpha_H\gamma_H]\\
        &=[\beta_H,(\alpha^0)_H]\\
        &=[(\beta^*)_{\N H P},(\alpha^0)_{\N H P}]\\
        &=[(\beta^*)_{\N H P}\gamma_{\N H P},\alpha_{\N H P}\gamma_{\N H P}]\\
        &=[f(\chi)_{\N H P}, \tau_{\N H P}]\\
        &=[f(\chi)_{\N H P}, \varphi_{\N H P}],
    \end{aligned}\]
    and we are done.
\end{proof}

Finally, we prove Corollary B.

\medskip
\begin{cor} 
  Let $G$ be $p$-solvable and let $\chi \in \irrp G$. Then $d_{\chi 1}$ equals either $1$ or $0$, and
  the number of $\chi \in \irrp G$ with $d_{\chi 1}=1$ is the number of $\norm GP$-orbits on
  $P/P'$, where $P \in \syl p G$.
  \end{cor}
  \begin{proof}
  By Theorem A, we may assume that $P\nor G$. Let
  $H$ be a $p$-complement of $G$. Let $\chi \in \irrp G$, let $\lambda \in \irr P$ be
  under $\chi$.
  Let $U=G_\lambda$ be the stabilizer of $\lambda$ in $G$, and  let $\hat\lambda \in \irr U$
  be the canonical extension of $\lambda$ to $U$ (see Corollary 6.27 of \cite{Is}).
  By the Gallagher correspondence, there is $\alpha \in \irr{U/P}$ such that $\chi=(\hat\lambda \alpha)^G$.
  Using that $o(\hat\lambda)$ is a $p$-power and Mackey's theorem, we have that
  $$d_{\chi 1}=[\chi_H, 1_H]=[\alpha_{U\cap H}, 1_{U\cap H}] \, .$$
  Since $\alpha$ has $P$ in its kernel, by Lemma \ref{restBij}, we have that $\alpha_{U\cap H}$ is irreducible. Therefore $d_{\chi 1}=1$ if and only if $\alpha=1$. Hence, the number of irreducible
  characters of $G$ with $d_{\chi 1}=1$ is in bijection with the number of $\norm GP$-orbits on $\irr{P/P'}$. As $P/P'$ is abelian, by Brauer's lemma on character tables (see Theorem 6.32 of \cite{Is}) it follows that the $\N G P$-actions over $\irr{P/P'}$ and $P/P'$ possess the same permutation character, and hence the number of $\N G P$-orbits on $\irr{P/P'}$ and $P/P'$ is the same.
  \end{proof}

\end{document}